\documentclass[12pt,reqno]{amsart}
\usepackage{amsmath, amssymb, amsthm}
\usepackage{url}
\usepackage[breaklinks]{hyperref}

\renewcommand{\Im}{\operatorname{Im}}

\renewcommand{\Re}{\operatorname{Re}}
\renewcommand{\Im}{\operatorname{Im}}

\renewcommand{\(}{\left\(}
\renewcommand{\)}{\right\)}
\renewcommand{\[}{\left\[}
\renewcommand{\]}{\right\]}

\numberwithin{equation}{section}
 \theoremstyle{plain}
\newtheorem{theorem}{Theorem}[section]
\newtheorem{lemma}[theorem]{Lemma}

   \makeatletter
\def\proof{\@ifnextchar[{\@oproof}{\@nproof}}
\def\@oproof[#1][#2]{\trivlist\item[\hskip\labelsep\textit{#2 Proof of\
#1.}~]\ignorespaces}
\def\@nproof{\trivlist\item[\hskip\labelsep\textit{Proof.}~]\ignorespaces}

\makeatother

\begin{document}
\title[Riesz-type criteria for the Riemann hypothesis]{Riesz-type criteria for the Riemann hypothesis }

\author{Archit Agarwal}
\address{Archit Agarwal\\ Department of Mathematics \\
Indian Institute of Technology Indore \\
Simrol,  Indore,  Madhya Pradesh 453552, India.} 
\email{archit.agrw@gmail.com,   phd2001241002@iiti.ac.in }

\author{Meghali Garg}
\address{Meghali Garg \\ Department of Mathematics \\
Indian Institute of Technology Indore \\
Simrol,  Indore,  Madhya Pradesh 453552, India.} 
\email{meghaligarg.2216@gmail.com,   phd2001241005@iiti.ac.in}

 \author{Bibekananda Maji}
\address{Bibekananda Maji\\ Department of Mathematics \\
Indian Institute of Technology Indore \\
Simrol,  Indore,  Madhya Pradesh 453552, India.} 
\email{bibek10iitb@gmail.com,  bibekanandamaji@iiti.ac.in}

\thanks{2010 \textit{Mathematics Subject Classification.} Primary 11M06; Secondary 11M26 .\\
\textit{Keywords and phrases.} Riemann zeta function,  Non-trivial zeros,  M\"{o}bious function, Riemann hypothesis.}

\maketitle

\begin{abstract}
In 1916,  Riesz proved that the Riemann hypothesis is equivalent to the bound
$
\sum_{n=1}^\infty \frac{\mu(n)}{n^2} \exp\left( - \frac{x}{n^2} \right) = O_{\epsilon} \left( x^{-\frac{3}{4} + \epsilon} \right),   
$
as $x \rightarrow\infty$, for any $\epsilon >0$.
Around the same time,  Hardy and Littlewood gave another equivalent criteria for the Riemann hypothesis while correcting an identity of Ramanujan.  In the present paper,  we establish a one-variable generalization of the identity of Hardy and Littlewood and as an application, we provide Riesz-type criteria for the Riemann hypothesis.  In particular,  we obtain the bound given by Riesz as well as the bound of Hardy and Littlewood. 
\end{abstract}

\section{introduction}

Ramanujan discovered many beautiful identities in his short life span of 32 years,  and most of the identities are correct, but a few are incorrect. Over the years, these erroneous identities also inspired mathematicians to develop many interesting results.  One of such incorrect identities is the following identity:
For any  $x>0$,
\begin{align}\label{Ramanujan_mu}
\sum_{n=1}^{\infty} \frac{\mu(n)}{n} \exp\left({-\frac{x}{n^2}}\right) = \sqrt{\frac{\pi}{x}} \sum_{n=1}^{\infty} \frac{\mu(n)}{n} \exp\left(- \frac{\pi^2}{n^2 x} \right),
\end{align}
where $\mu(n)$ denotes the M\"{o}bius function.  Ramanujan, during his stay at Trinity, showed the above identity to Hardy and Littlewood from his Second Notebook \cite[p.~312]{Rama_2nd_Notebook}.  Unfortunately,  all the identities on page 312 are incorrect. Bruce Berndt\footnote{There is a typo in \cite[Equation (37.3), p.~470]{BCB-V},  in which the expression $\frac{\pi}{p}$ in the right-side sum over $\rho$ should be replaced by $\frac{\pi}{\sqrt{p}}$. } \cite[p.~468-469]{BCB-V} gave a very nice numerical explanation about the wrongness of the identity  \eqref{Ramanujan_mu}.
 Hardy and Littlewood  \cite[p.~156, Section 2.5]{HL-1916}  established a corrected version of \eqref{Ramanujan_mu}.  Mainly,  they proved that,  for any $x>0$,  
 \begin{align}\label{Hardy-Littlewood}
 \sum_{n=1}^{\infty} \frac{\mu(n)}{n} \exp\left({-\frac{x}{n^2}}\right) = \sqrt{\frac{\pi}{x}} \sum_{n=1}^{\infty} \frac{\mu(n)}{n} \exp\left(- \frac{\pi^2}{n^2 x} \right)- \frac{1}{2 \sqrt{\pi}} \sum_{\rho} \left(  \frac{\pi}{\sqrt{x}} \right)^{\rho}  \frac{ \Gamma\left(\frac{1-\rho}{2} \right) }{\zeta'(\rho)},
 \end{align}
where the right-side sum over $\rho$ runs through the non-trivial zeros of the Riemann zeta function $\zeta(s)$.  The identity \eqref{Hardy-Littlewood} holds under the assumption that, all non-trivial zeros of $\zeta(s)$ are simple.  Replacing $x$ by $\alpha^2$ and $\beta = \pi/\alpha$,  one can immediately show that
\begin{align}\label{Rama_Hardy_Little}
\sqrt{\alpha} \sum_{n=1}^{\infty} \frac{\mu(n)}{n} \exp\left({-\left(\frac{\alpha}{n}\right)^2}\right) - \sqrt{\beta} \sum_{n=1}^{\infty} \frac{\mu(n)}{n} \exp\left({-\left(\frac{\beta}{n}\right)^2}\right) = -\frac{1}{2\sqrt{\beta}} \sum_{\rho} \frac{ \Gamma\left(\frac{1-\rho}{2} \right) \beta^\rho}{\zeta'(\rho)}.
\end{align}
The convergence of the above right-side sum is delicate.  
Hardy and Littlewood showed that this infinite sum over the non-trivial zeros of $\zeta(s)$ is convergent under the assumption of bracketing the terms, that is,   the terms corresponding to the non-trivial zeros $\rho_1$ and $\rho_2$ will be inside the same bracket if for some positive constant $A_0$, we have
\begin{align}\label{bracketing}
|\Im(\rho_1) - \Im(\rho_2)| < \exp \left( -\frac{A_0 \Im(\rho_1)}{\log(\Im(\rho_1) )} \right) + \exp \left( -\frac{A_0 \Im(\rho_2)}{\log(\Im(\rho_2) )} \right).
\end{align}

To know more about the identity \eqref{Rama_Hardy_Little},  readers can see Berndt \cite[p.~470]{BCB-V}, Paris and Kaminski \cite[p.~143]{PK}, and Titchmarch \cite[p.~219]{Tit}.  

Riesz \cite{Riesz},  in 1916,  proved that the Riemann hypothesis is equivalent to the fact that
\begin{equation}\label{Riesz}
P_2(x):=  \sum_{n=1}^\infty \frac{\mu(n)}{n^2} \exp\left( - \frac{x}{n^2} \right) = O_{\epsilon} \left( x^{-\frac{3}{4} + \epsilon} \right), \quad {\rm as}\,\,  x \rightarrow \infty,
\end{equation}
for any positive $\epsilon$.  Around the same time,  as an application of \eqref{Rama_Hardy_Little},  Hardy and Littlewood \cite[p.~161]{HL-1916} showed that the Riemann hypothesis is equivalent to the following bound
\begin{align}\label{Riesz type_Hardy_Littlewood}
P_1(x):= \sum_{n=1}^\infty  \frac{\mu(n)}{n} \exp\left({-\frac{x}{n^2}}\right) = O_{\epsilon}\left( x^{-\frac{1}{4}+ \epsilon } \right), \quad \mathrm{as}\,\, x \rightarrow \infty,
\end{align}
for any positive $\epsilon$.

In the current paper,  we establish a one-variable generalization of the identity \eqref{Hardy-Littlewood} and as an application of our result,  we obtain an equivalent criteria for the Riemann hypothesis.  In particular,  we are able to recover the equivalent criteria \eqref{Riesz} given by Riesz,  as well as the bound \eqref{Riesz type_Hardy_Littlewood} given by Hardy and Littlewood. 

Now we define an important special function.
Let $c_1, \cdots, c_p$ and $d_1, \cdots, d_q$ be  $p+q$ complex numbers.  The generalized hypergeometric series \cite[p. 404, Equation 16.2.1]{NIST} 
 is defined by the following series expansion
\begin{align*}
{}_pF_q \left( c_1, \cdots, c_p;\,  d_1, \cdots, d_q;\, z \right):= \sum_{n=0}^{\infty} \frac{(c_1)_n \cdots (c_p)_n}{(d_1)_n\cdots (d_q)_n} \frac{z^n}{n!},
\end{align*}
where $(c)_n:=\frac{\Gamma(c+n)}{\Gamma(c)}$.  This series converges for all complex values of $z$ if  $p \leq q$.  However, if $p=q+1$,  then it converges for $|z|<1$,   and in this case,  it can be analytically continued to $ \mathbb{C}$ if we consider a branch cut from $1$ to $+\infty$.  In particular,  when $p=q=1$,  that is,  ${}_1F_1 \left( c_1 ;\,  d_1;\, z \right)$ is an entire function and if $c_1=d_1$, it reduces to the exponential function.

Now we are ready to state the main identity.

\begin{theorem}\label{first_generalization}
Let $k \geq 1$ be a real number.  Assume that all non-trivial zeros of $\zeta(s)$ are simple. Then for any positive $x$, we have
\begin{align}\label{AGM}
    \sum_{n=1}^{\infty} \frac{\mu(n)}{n^k} \exp\left({-\frac{x}{n^2}}\right) =\frac{\Gamma(\frac{k}{2})}{x^\frac{k}{2}}\sum_{n=1}^{\infty}\frac{\mu(n)}{n}{}_1F_1 \left(\frac{k}{2}; \frac{1}{2}; - \frac{\pi^2}{n^2x} \right) + \frac{1}{2}\sum_{\rho}\frac{\Gamma(\frac{k-\rho}{2})}{\zeta'(\rho)}x^{-\frac{(k-\rho)}{2}},
\end{align}
where the sum over $\rho$ runs through all non-trivial zeros of $\zeta(s)$ and satisfying the condition\eqref{bracketing}. 
\end{theorem}
Substituting $k=1$ in \eqref{AGM},  and replacing $x$ by $\pi^2/x$,  one can immediately obtain \eqref{Hardy-Littlewood}.  
Letting $x \rightarrow \infty $ in \eqref{AGM} and assuming the absolute convergence of the right side series over the non-trivial zeros of $\zeta(s)$,  one can immediately see that 
\begin{align*}
P_{k}(x):= \sum_{n=1}^{\infty} \frac{\mu(n)}{n^k} \exp\left({-\frac{x}{n^2}}\right) = O \bigg(x^{-\frac{k}{2}+\frac{1}{4}}\bigg), 
\end{align*}
under the assumption of the Riemann hypothesis. 
This observation suggests us to obtain the following equivalent criteria for the Riemann hypothesis. 
\begin{theorem}\label{Agarwal_Garg_Maji bound}
The Riemann hypothesis is equivalent to the bound
\begin{equation}\label{AGM_bound}
P_{k}(x) = O_{\epsilon} \bigg(x^{-\frac{k}{2}+\frac{1}{4} + \epsilon }\bigg), \quad \mathrm{as}\,\, x \rightarrow \infty,
\end{equation}
for any positive $\epsilon$. 
\end{theorem}
In particular,  letting $k=1$ in \eqref{AGM_bound} allows us to recover the bound \eqref{Riesz type_Hardy_Littlewood} of Hardy and Littlewood,  and $k=2$ gives the bound \eqref{Riesz} of Riesz.

\section{Preliminaries}

In this section,  we state a few well-known results,  which will play a vital role in proving our main identity.  We begin with an important summation formula, namely, Euler's summation formula.
 \begin{lemma}
 Let $\{ a_n\}$ be a sequence of complex number and $f(t)$ be a continuously differentiable function on $[1,x]$.  Consider $A(x):= \sum_{1 \leq n \leq x} a_n$.  Then we have
 \begin{align}\label{Euler's summation}
 \sum_{ 1\leq n \leq x} a_n f(n) = A(x) f(x) - \int_{1}^{x} A(t) f'(t) {\rm d}t.
 \end{align}
 \end{lemma}
\begin{proof}
Proof of this result can be found in \cite[p.~17]{Murty}
\end{proof}

The next result gives an important information about the asymptotic expansion of the gamma function. 
\begin{lemma}\label{Stirling}
In a vertical strip $\alpha \leq \sigma \leq \beta$,  we have
\begin{equation}\label{Stirling_equn}
|\Gamma (\sigma + i T)| = \sqrt{2\pi} | T|^{\sigma - 1/2} e^{-\frac{1}{2} \pi |T|} \left(1 + O\left(\frac{1}{|T|}\right)  \right) \quad {\rm as} \quad |T|\rightarrow \infty.
\end{equation}
 \end{lemma}
 This is popularly known as Stirling's formula.  
Now we will state an important result, which will provide us a bound for $1/\zeta(s)$. 
\begin{lemma}\label{bound_1by_zeta(s)}
Suppose, for every  non-trivial zero $\rho$ of $\zeta(s)$,  there exist a sequence of arbitary larger positive numbers $T$ with $|T-\Im(\rho)| >
 \exp\left(-\frac{C_0 \Im(\rho) }{\log(\Im(\rho))}\right)$, where $A$ is  some positive constant. Then,
 \begin{equation*}
  \frac{1}{|\zeta(\sigma + i T)|} < e^{C_1 T},
 \end{equation*}
 for some constant $0<C_1 < \pi/4$.
 \end{lemma}
 \begin{proof}
 One can find a proof of this result in \cite[p.~219]{Tit}.  An important point to note that  the assumption of the bracketing condition \eqref{bracketing} guarantee the existence of such a sequence. 
 \end{proof}
 The next result will be one of the crucial ingredients in providing an equivalence criteria for the Riemann hypothesis.
 \begin{lemma}\label{Mellin_P_k(x)}
  Let $k \geq 1$. 
In the region $\frac{1-k}{2} <\Re(s) < 1$, except at $s=0$,  we have
\begin{align*}
\int_{0}^{\infty} x^{-s-1} P_k(x) {\rm d}x = \frac{\Gamma(-s)}{\zeta(2s+k)}.
\end{align*}
\end{lemma}
\begin{proof}
This lemma was proved by Riesz \cite{Riesz} for $k=1$,  which was used by Hardy and Littlewood \cite[p.~161,  Equn (2.544)]{HL-1916} to give an equivalent criteria for the Riemann hypothesis. Here we shall prove this lemma for $k >1$.    Using the definition of $P_k(x)$,  one can show that
\begin{align}\label{another form_P_k(x)}
P_{k}(x) = \sum_{n=1}^{\infty} \frac{\mu(n)}{n^k} \exp\left({-\frac{x}{n^2}}\right)  = \sum_{n=1}^{\infty} \frac{\mu(n)}{n^k} \sum_{m=0}^\infty \frac{(-1)^m x^m }{m! n^{2m} } 
 = \sum_{m=0}^\infty \frac{(-1)^m x^m }{m! \zeta(k+2m) } .
\end{align}
Note that, when $k=1$, the last sum will start from $m=1$ since the term corresponding to $m=0$ will become zero. 
For $\Re(s) > \frac{1-k}{2}$,  use the series expansion of $\zeta(2s +k)$ to write
\begin{align}\label{Mellin transform of P_k(x)}
\zeta(2s+k) \int_{0}^{\infty} x^{-s-1} P_k(x) {\rm d}x  & = \sum_{n=1}^{\infty} \frac{1}{n^k} \int_{0}^{\infty}  \frac{x^{-s-1}}{n^{2s}} P_k(x) {\rm d}x  \nonumber \\
& = \sum_{n=1}^{\infty} \frac{1}{n^k}  \int_{0}^{\infty}  x^{-s-1} P_k\left(\frac{x}{n^2} \right) {\rm d}x \nonumber \\
& =  \int_{0}^{\infty}  x^{-s-1} \sum_{n=1}^{\infty} \frac{1}{n^k} P_k\left(\frac{x}{n^2} \right) {\rm d}x.
\end{align}
Here in the second last step we replaced $x$ by $x/n^2$ and in the final step interchanging of summation and integration was possible due to Weierstrass M-test as one can easily show that  $\left| \frac{1}{n^k} P_k\left(\frac{x}{n^2} \right) \right| \leq \frac{\zeta(k)}{n^k}$ for $k >1$.   Now using \eqref{another form_P_k(x)}, one can find that 
\begin{align}\label{exp(-x)}
\sum_{n=1}^{\infty} \frac{1}{n^k} P_k\left(\frac{x}{n^2} \right)  =  \sum_{n=1}^{\infty} \frac{1}{n^k}   \sum_{m=0}^\infty \frac{(-1)^m x^m }{m! n^{2m} \zeta(k+2m) } 
 =\sum_{m=0}^\infty \frac{(-1)^m x^m }{m! } = e^{-x}.
\end{align}
On substituting \eqref{exp(-x)} in \eqref{Mellin transform of P_k(x)}, we get
\begin{align*}
\zeta(2s+k) \int_{0}^{\infty} x^{-s-1} P_k(x) {\rm d}x =  \int_{0}^{\infty}  x^{-s-1} e^{-x} {\rm d} x  = \Gamma(-s),
\end{align*}
for $\Re(s)<0$.  By analytic continuation it can be extended to the right half plane  $\Re(s)<1$
except at $s=0$.

\end{proof}
Now we define an important special function,  namely,  the Meijer $G$-function \cite[p.~415, Definition 16.17]{NIST}, which is the generalization of many well-known special functions.
Let $m,n,p,q$ be non-negative integers with $0\leq m \leq q$, $0\leq n \leq p$. Let $a_1, \cdots, a_p$ and $b_1, \cdots, b_q$ be complex numbers with $a_i - b_j \not\in \mathbb{N}$ for $1 \leq i \leq n$ and $1 \leq j \leq m$. Then the Meijer $G$-function is defined by 
\begin{align}\label{Meijer-G}
G_{p,q}^{\,m,n} \!\left(  \,\begin{matrix} a_1,\cdots , a_p \\ b_1, \cdots , b_q \end{matrix} \; \Big| z   \right) := \frac{1}{2 \pi i} \int_L \frac{\prod_{j=1}^m \Gamma(b_j - s) \prod_{j=1}^n \Gamma(1 - a_j +s) z^s  } {\prod_{j=m+1}^q \Gamma(1 - b_j + s) \prod_{j=n+1}^p \Gamma(a_j - s)}\mathrm{d}s,
\end{align}
where the line of integration $L$, from $-i \infty$ to $+i \infty$,   separates the poles of the factors $\Gamma(1-a_j+s)$  from those of the factors $\Gamma(b_j-s)$.   The above integral converges if $p+q < 2(m+n)$ and $|\arg(z)| < \left(m+n - \frac{p+q}{2} \right) \pi$. 

Now we state Slater's theorem \cite[p.~415, Equation 16.17.2]{NIST}, which will allow us to express Meijer $G$-function in terms of generalized hypergeometric functions. 
If $p \leq q$ and $ b_j - b_k \not\in \mathbb{Z}$ for $j\neq k$, $1 \leq j, k \leq m$, then 
\begin{align}\label{Slater}
& G_{p,q}^{\,m,n} \!\left(   \,\begin{matrix} a_1, \cdots , a_p \\ b_1, \cdots , b_q \end{matrix} \; \Big| z   \right)  \\
& \quad = \sum_{k=1}^{m} A_{p,q,k}^{m,n}(z) {}_p F_{q-1} \left(  \begin{matrix}
1+b_k - a_1,\cdots, 1+ b_k - a_p \\
1+ b_k - b_1, \cdots, *, \cdots, 1 + b_k - b_q 
\end{matrix} \Big| (-1)^{p-m-n} z  \right), \nonumber 
\end{align}
where $*$ means that the entry $1 + b_k - b_k$ is omitted and 
\begin{align*}
A_{p,q,k}^{m,n}(z) := \frac{ z^{b_k}  \prod_{ j=1,  j\neq k}^{m} \Gamma(b_j - b_k ) \prod_{j=1}^n   \Gamma( 1 + b_k -a_j ) }{ \prod_{j=m+1}^{q} \Gamma(1 + b_k - b_{j}) \prod_{j=n+1}^{p} \Gamma(a_{j} - b_k)  }.
\end{align*}
Now we are ready to provide the proof of the main results.

\section{Proof of main results}

\begin{proof}[ Theorem {\rm \ref{first_generalization}}][]
It is well-known that the gamma function $\Gamma(s)$ is the Mellin transform of $e^{-x}$.  Therefore,  $e^{-x}$ is the inverse Mellin transform of $\Gamma(s)$.  That is,
\begin{align*}
e^{-x} = \frac{1}{2 \pi i}  \int_{c-i\infty}^{c+i \infty}\Gamma(s) x^{-s} {\rm d}s,
\end{align*}
valid for any $c>0$.  If we shift the line of integration to the line $c\in (-1,0)$,  then we will have
\begin{align}\label{inverse_Mellin}
e^{-x} - 1 = \frac{1}{2 \pi i}  \int_{c-i\infty}^{c+i \infty}\Gamma(s) x^{-s} {\rm d}s. 
\end{align}
First, we recall that the prime number theorem is equivalent to the fact that $\sum_{n=1}^{\infty} \frac{\mu(n)}{n}=0$.  One can easily show that,  for any $k \in \mathbb{N}$,  the series  $\sum_{n=1}^{\infty} \frac{\mu(n)}{n^k} \exp \left({-\frac{x}{n^2}}\right)$ is convergent.  In fact,  for $k \geq 2$,  it is absolutely convergent,  whereas for $k=1$, it is conditionally convergent.  However,  our proof is valid for any $k \geq 1$.
With the help of \eqref{inverse_Mellin},  we can write 
\begin{align}
     \sum_{n=1}^{\infty} \frac{\mu(n)}{n^k} \exp \left({-\frac{x}{n^2}}\right)&=\sum_{n=1}^{\infty} \frac{\mu(n)}{n^k}- \sum_{n=1}^{\infty} \frac{\mu(n)}{n^k} \left(1-\exp\left({-\frac{x}{n^2}}\right)\right) \nonumber \\
     &=\sum_{n=1}^{\infty} \frac{\mu(n)}{n^k}+\sum_{n=1}^{\infty} \frac{\mu(n)}{n^k}\frac{1}{2\pi i} \int_{c-i\infty}^{c+i \infty}\Gamma(s) \left(\frac{x}{n^2}\right)^{-s} {\rm d}s  \nonumber \\
     & =  \sum_{n=1}^{\infty} \frac{\mu(n)}{n^k}+\frac{1}{2\pi i}\int_{c-i \infty}^{c+i \infty}\frac{\Gamma(s)}{\zeta(k-2s)} x^{-s} {\rm d}s, \label{main equation}
\end{align}
here in the ultimate step we interchanged the summation and integration since the series $\sum_{n=1}^{\infty} \frac{\mu(n)}{n^{k-2s}}$ is absolutely absolutely and uniformly convergent in any compact subset of the domain $\Re(k-2s) >1$.  Now we shall try to simplify the line integral present in \eqref{main equation}.  Let us denote this line integral as
\begin{align}\label{main_line_integration}
I_{k}(x):= \frac{1}{2\pi i}\int_{c-i \infty}^{c+i \infty} \frac{\Gamma(s)}{\zeta(k-2s)} x^{-s} {\rm d}s. 
\end{align}
First, we observe that the integrand function  has simple poles at non-positive integers due to the simple poles of $\Gamma(s)$.  Again,  for any $m \in \mathbb{N}$,  $\frac{k}{2}+m$ is also a simple pole of the integrand function since $\zeta(k-2s)$ has trivial zeros at $-2m$.  Another important observation is that $\zeta(k-2s)$ has infinitely many non-trivial zeros in the critical strip $\frac{k}{2} -1 < \Re(s) < \frac{k}{2}$.  If one assumes the Riemann hypothesis,  then we can say all the non-trivial zeros of $\zeta(k-2s)$ will lie on the line $\Re(s) = \frac{k}{2}-\frac{1}{4}$.  However,  to obtain our main identity,  we do not require the assumption of the Riemann hypothesis.   Nevertheless, as we are interested to collect the contributions of the residual terms corresponding to the non-trivial zeros of $\zeta(k-2s)$,  so we have to shift the line of integration to the line  $\Re(s) = \lambda$, where $\lambda \in \left( \frac{k}{2},  \frac{k}{2}+1 \right)$. 

Thus,  we  consider the contour $\mathcal{C}$ determined by the line segments $[ \lambda-iT,  \lambda+iT], [\lambda+i T, c + i T], [c + i T,  c - i T]$, and $ [c - i T, \lambda-iT]$,  where $ -1<c<0$ and $\frac{k}{2} < \lambda < \frac{k}{2}+1$,  and $T$ is some large positive number.   Now appealing to Cauchy's residue theorem,  we have
\begin{align}\label{CRT_application}
    \frac{1}{2\pi i}\int_{\mathcal{C}}\frac{  \Gamma(s)  }{\zeta(k-2s) } x^{-s} {\rm d}s = R_{0} + \mathcal{R}_{T}(x),
\end{align}
where $R_0$ denotes the residual term due to the simple pole of $\Gamma(s)$ at $s=0$ and $\mathcal{R}_{T}(x)$ denotes the sum of the residual terms corresponding to the non trivial zeroes  $\rho$ of $\zeta(k-2s)$ with $|\Im(\rho)| < T$.  Next,  we shall try  to show that the horizontal integrals
\begin{align*}
H_1(x; T):= \frac{1}{2 \pi i } \int_{\lambda + i T}^{c+i T} \frac{  \Gamma(s)  }{\zeta(k-2s) } x^{-s} {\rm d}s, \quad {\rm and} \quad
H_2(x; T):= \frac{1}{2 \pi i } \int_{c - i T}^{ \lambda - i T} \frac{  \Gamma(s)  }{\zeta(k-2s) } x^{-s} {\rm d}s  
\end{align*}
 vanish as $T$ approaches to infinity.  Here,  we employ Starling's formula for $\Gamma(s)$ i.e.,  Lemma \ref{Stirling_equn} and  along with Lemma \ref{bound_1by_zeta(s)},  we find that
 \begin{align*}
 |H_1(x; T)| \ll \int_{\lambda}^{c} |T|^C \exp\left( 2 C_1 T- \frac{\pi}{2} |T| \right) {\rm d}T,
 \end{align*}
where $C_1$ is some constant, and $0< C_1 <\pi/4$. This immediately implies that $H_1(x; T)$ vanishes as $T \rightarrow \infty$.  In a similar way,  one can show that $H_2(x;T)$ goes to zero as $T \rightarrow \infty$.  Now letting $T\rightarrow \infty$ in \eqref{CRT_application},  we arrive at
\begin{align}\label{two vertical integral}
\frac{1}{2 \pi i } \int_{c - i T}^{c+i T} \frac{  \Gamma(s)  }{\zeta(k-2s) } x^{-s} {\rm d}s= \frac{1}{2 \pi i } \int_{\lambda - i T}^{\lambda+i T} \frac{  \Gamma(s)  }{\zeta(k-2s) } x^{-s} {\rm d}s - R_0 - \mathcal{R}(x),
\end{align}
where $\mathcal{R}(x)$ is an infinite series,  that is,  the sum of the residual terms corresponding to the non-trivial zeros of $\zeta(k-2s)$. 
One can easily calculate the residue at $s=0$,  
\begin{align}\label{Residue_s=0}
R_{0} = \lim_{s \rightarrow 0}  \frac{  s \Gamma(s)  }{\zeta(k-2s) } x^{-s}= \frac{1}{\zeta(k)}. 
\end{align}
If we assume that all the non-trivial zeros of $\zeta(s)$ are simple,  then we get
\begin{align}\label{infinite_residual term}
\mathcal{R}(x)= \sum_{ \rho } \lim_{ s \rightarrow \frac{k-\rho}{2}} \frac{ \left( s-  \frac{k-\rho}{2} \right) \Gamma(s)   }{\zeta(k-2s)  } x^{-s} = -\frac{1}{2}  \sum_{ \rho } \frac{\Gamma(\frac{k-\rho}{2})}{\zeta'(\rho)} x^{-\frac{k-\rho}{2}},
\end{align}
where the sum runs through all the non-trivial zeros $\zeta(s)$.  More generally,  if we do not assume the simplicity hypothesis,  that is, if $n_\rho$ is the multiplicity of $\rho$, then
\begin{align*}
\mathcal{R}(x)= \sum_{\rho} \frac{1}{(n_\rho - 1)!} 
&   \lim_{s\rightarrow  \frac{k-\rho}{2}  } \frac{ {\rm d}^{n_\rho -1}}{{\rm d}s^{n_\rho -1 } } \Bigg\{ \left(s-  \frac{k-\rho}{2}\right)^{n_\rho} \frac{\Gamma(s) }{\zeta(k-2s)x^s} \Bigg\}.
\end{align*}
Substituting \eqref{Residue_s=0} and \eqref{infinite_residual term} in \eqref{two vertical integral} and combining with \eqref{main equation} and \eqref{main_line_integration},  we obtain
\begin{align}\label{final series_with residual}
 \sum_{n=1}^{\infty} \frac{\mu(n)}{n^k} \exp \left({-\frac{x}{n^2}}\right) =  \frac{1}{2 \pi i } \int_{\lambda - i \infty}^{\lambda+i \infty} \frac{  \Gamma(s)  }{\zeta(k-2s) } x^{-s} {\rm d}s  + \frac{1}{2}  \sum_{ \rho } \frac{\Gamma(\frac{k-\rho}{2})}{\zeta'(\rho)} x^{-\frac{k-\rho}{2}}.
\end{align}
At this juncture,  we shall try to simplify the right vertical integral 
\begin{align}\label{right vertical integral}
V(x, k):=   \frac{1}{2 \pi i } \int_{\lambda - i \infty}^{\lambda+i \infty} \frac{  \Gamma(s)  }{\zeta(k-2s) } x^{-s} {\rm d}s,
\end{align}
where $\frac{k}{2} < \lambda < \frac{k}{2}+1$.   Here we make use of the symmetric form of the functional equation of $\zeta(s)$, that is,
\begin{equation*}
\pi^{-\frac{s}{2}}\Gamma\left(\frac{s}{2}\right)\zeta(s)=\pi^{-\frac{1-s}{2}}\Gamma\left(\frac{1-s}{2}\right)\zeta(1-s),
\end{equation*}
replace $s$ by $k-2s$ and simplify to see
\begin{align}\label{Zeta (k-2s)}
    \frac{1}{\zeta(k-2s)} = \frac{\pi^{\frac{1}{2}-k+2s} \Gamma\left(\frac{k-2s}{2} \right)}{\Gamma\left(\frac{1-k+2s}{2} \right) \zeta(1-k+2s)}.
\end{align}
Plugging \eqref{Zeta (k-2s)} in \eqref{right vertical integral},  one can see that
\begin{align}\label{second form_V(x,k)}
V(x, k):=   \frac{\pi^{ \frac{1}{2}-k}}{2 \pi i } \int_{\lambda - i \infty}^{\lambda+i \infty}  \frac{ \Gamma(s) \Gamma\left(\frac{k}{2}-s \right)}{\Gamma\left(\frac{1-k}{2} +s \right) \zeta(1-k+2s)} \left( \frac{x}{\pi^2}  \right)^{-s} {\rm d}s.
\end{align}
Here we can check that $1<\Re(1-k+2s)<3 $ since $\frac{k}{2} < \Re(s)= \lambda < \frac{k}{2}+1$.  Thus,  we can write 
\begin{align*}
\frac{1}{\zeta(1-k+2s)} = \sum_{n=1}^{\infty} \frac{\mu(n) n^{k-1}}{n^{2s}}.
\end{align*}
Making use of the above series expansion in \eqref{second form_V(x,k)} and then interchanging the summation and integration,  \eqref{second form_V(x,k)} reduces to 
\begin{align}\label{thrid form_V(x,k)}
V(x, k):=   \pi^{ \frac{1}{2}-k} \sum_{n=1}^\infty \mu(n) n^{k-1}  I(X_n, k),
\end{align}
where 
\begin{align}\label{I(X_n, k)}
I(X_n, k):=  \frac{1}{2 \pi i } \int_{\lambda - i \infty}^{\lambda+i \infty}   \frac{ \Gamma(s) \Gamma\left(\frac{k}{2}-s \right)}{\Gamma\left(\frac{1-k}{2} +s \right) }  X_n^{-s} {\rm d}s,
\end{align}
and $X_n = \frac{n^2 x}{\pi^2}$. 
Now our main goal is the simplify the integral $I(X_n, k)$.  We shall try to express this integral in terms of the Meijer $G$-function.  We know that $\Gamma(s)$ has poles at $s=0,-1,-2, \cdots$,  and the poles of $\Gamma\left( \frac{k}{2} -s  \right)$ are at $\frac{k}{2},  \frac{k}{2}-1,  \frac{k}{2} -2,  \cdots$.  Thus, we can easily observe that the line of integration $\Re(s)=\lambda$ 
does not separate the poles of $\Gamma(s)$ with the poles of $\Gamma\left( \frac{k}{2} -s  \right)$
since $\frac{k}{2} < \lambda < \frac{k}{2}+1$.  Therefore,  we shift the line of integration $\Re(s)= \lambda$ to a new line $\Re(s)= \lambda_1$,  where $0< \lambda_1 < \frac{k}{2}$,  so that this new line of integration does separate the poles $\Gamma(s)$ from the poles of $\Gamma\left(\frac{k}{2} -s  \right)$.  With the help of this new line of integration,  we construct a contour $\mathcal{C}_1$ determined by line segments $[\lambda-iT,\lambda+iT],  [\lambda+i T,  \lambda_1+iT], [\lambda_1+iT, \lambda_1-iT]$,  and $[\lambda_1- iT,  \lambda-iT]$.  Again,  utilizing the residue theorem,  we have
\begin{align}\label{2nd_application_CRT}
\frac{1}{2 \pi i}   \int_{\lambda-iT}^{\lambda+iT} + \int_{\lambda+i T}^{\lambda_1 + i T} + \int_{\lambda_1 + i T}^{\lambda_1 - i T} + \int_{\lambda_1 - i T}^{\lambda - i T}  \frac{ \Gamma(s) \Gamma\left(\frac{k}{2}-s \right)}{\Gamma\left(\frac{1-k}{2} +s \right) }  X_n^{-s} {\rm d}s = \mathop{\rm{Res}}_{s=\frac{k}{2}}  \frac{ \Gamma(s) \Gamma\left(\frac{k}{2}-s \right)}{\Gamma\left(\frac{1-k}{2} +s \right) }  X_n^{-s}.
\end{align}
Letting $T \rightarrow \infty$ and using Stirling's formula \eqref{Stirling_equn} for $\Gamma(s)$,  one can show that both of the horizontal integrals vanish.  Therefore,  \eqref{2nd_application_CRT} becomes
\begin{align}\label{another_right vertical}
\frac{1}{2 \pi i}   \int_{\lambda-i\infty}^{\lambda+i \infty}  \frac{ \Gamma(s) \Gamma\left(\frac{k}{2}-s \right)}{\Gamma\left(\frac{1-k}{2} +s \right) }  X_n^{-s} {\rm d}s =  \frac{1}{2 \pi i}   \int_{\lambda_1 - i \infty}^{\lambda_1 - i \infty}    \frac{ \Gamma(s) \Gamma\left(\frac{k}{2}-s \right)}{\Gamma\left(\frac{1-k}{2} +s \right) }  X_n^{-s} {\rm d}s -\frac{\Gamma(\frac{k}{2})X_n^{-\frac{k}{2}}}{\sqrt{\pi}}.
\end{align}
We have already seen that the line of integration $\Re(s)= \lambda_1$ does separate the poles of $\Gamma(s)$ from the poles of $\Gamma\left(\frac{k}{2}-s \right)$.  Thus,  making use of the definition \eqref{Meijer-G} of the Meijer $G$-function,  with $m=n=p=1,  q=2$,  and $a_1=1,  b_1= \frac{k}{2},  b_2= \frac{1+k}{2}$,  we can write
\begin{align}\label{in terms_Meijer-G}
\frac{1}{2 \pi i}   \int_{\lambda_1 - i \infty}^{\lambda_1 - i \infty}    \frac{ \Gamma(s) \Gamma\left(\frac{k}{2}-s \right)}{\Gamma\left(\frac{1-k}{2} +s \right) }  X_n^{-s} {\rm d}s =  G_{1,2}^{1,1} \left(\begin{matrix} 1 \\ \frac{k}{2},\frac{k+1}{2} \end{matrix} \Big| \frac{1}{X_n}\right).
\end{align}
Here one can easily check all the necessary conditions for the convergence of this Meijer $G$-function.  Now invoking Slater's theorem \eqref{Slater},  one can show that
\begin{align}\label{MeijerG_1F1}
G_{1,2}^{1,1} \left(\begin{matrix} 1 \\ \frac{k}{2},\frac{k+1}{2} \end{matrix} \Big| \frac{1}{X_n}\right) = \frac{\Gamma\left( \frac{k}{2} \right)}{X_n^{ \frac{k}{2}} \sqrt{\pi}} {}_1F_{1} \left( \frac{k}{2}; \frac{1}{2}; - \frac{1}{X_n} \right).
\end{align}
Substituting \eqref{MeijerG_1F1} in \eqref{in terms_Meijer-G} and together with \eqref{another_right vertical} and \eqref{I(X_n, k)},  we see that the right vertical integral $V(x,k)$ \eqref{thrid form_V(x,k)} takes the final shape as 
\begin{align*}
V(x, k) =  \frac{\Gamma\left( \frac{k}{2} \right)}{\pi^k}   \sum_{n=1}^\infty \frac{ \mu(n) n^{k-1} }{X_n^{ \frac{k}{2}} }  \left\{  {}_1F_{1} \left( \frac{k}{2}; \frac{1}{2}; - \frac{1}{X_n} \right) -1  \right\}.
\end{align*}
Now plugging $X_n = \frac{n^2 x}{\pi^2 }$,  the above expression reduces to 
\begin{align}
V(x, k)  & =  \frac{\Gamma\left( \frac{k}{2} \right)}{x^{\frac{k}{2}}}   \sum_{n=1}^\infty \frac{ \mu(n) }{n }  \left\{  {}_1F_{1} \left( \frac{k}{2}; \frac{1}{2}; - \frac{\pi^2}{n^2 x} \right) -1  \right\}  \label{with -1}\\
& = \frac{\Gamma\left( \frac{k}{2} \right)}{x^{\frac{k}{2}}}   \sum_{n=1}^\infty \frac{ \mu(n) }{n }  \left\{  {}_1F_{1} \left( \frac{k}{2}; \frac{1}{2}; - \frac{\pi^2}{n^2 x} \right)  \right\}. \label{without -1 term}
\end{align}
In the last step,  we utilized the fact that $\sum_{n=1}^\infty \frac{\mu(n)}{n}=0$,  which is equivalent to  the prime number theorem.  Here we note that the expression \eqref{with -1} is more effective than the expression \eqref{without -1 term}  to check numerically since the series \eqref{with -1} is absolutely convergent. 
Finally,  in view of \eqref{without -1 term} and \eqref{right vertical integral}, and together with \eqref{final series_with residual},  we complete the proof of Theorem \ref{first_generalization}.

\end{proof}

\begin{proof}[Theorem {\rm \ref{Agarwal_Garg_Maji bound}}][] First,  let us assume that the Riemann hypothesis is true.  Under the assumption of the Riemann hypothesis,  Littlewood  showed that,  for any $\epsilon >0$, 
\begin{align}\label{summatory_Mobius}
A(x):= \sum_{1 \leq n \leq x} \mu(n) = O_{\epsilon}(x^{\frac{1}{2}+\epsilon}). 
\end{align}
We invoke Euler's partial summation formula \eqref{Euler's summation}, with $a_n = \mu(n)$ and $f(t)= t^{-k}$,  to see that
\begin{align}\label{M(l,n)}
M(\ell;n) := \sum_{m=\ell}^{n} \frac{\mu(m)}{m^k}  & =  A(n) f(n)- A(\ell-1) f(\ell-1) - \int_{\ell-1}^n A(t) f'(t) {\rm d}t.
\end{align}
Making use of \eqref{summatory_Mobius} in \eqref{M(l,n)} and simplifying,  one can easily show that 
\begin{align}\label{bound for M(l,n)}
M(\ell;n) = O_{\epsilon}(\ell^{\frac{1}{2}-k+\epsilon}), 
\end{align}
uniformly in $n$.  Our main aim is to find an estimate for the following sum 
\begin{align*}
P_{k}(x)= \sum_{n=1}^{\infty} \frac{\mu(n)}{n^k} \exp\left({-\frac{x}{n^2}}\right).
\end{align*}
We separate the sum into two parts,  that is,  
\begin{align}\label{P_k(x^2) interms of S_1 and S_2}
P_{k}(x^2) & =  \sum_{n=1}^{\ell-1}  \frac{\mu(n)}{n^k} \exp\left({-\frac{x^2}{n^2}}\right) + \sum_{n=\ell}^{\infty}  \frac{\mu(n)}{n^k} \exp\left({-\frac{x^2}{n^2}}\right) \nonumber \\
 & := S_1(x^2) + S_2(x^2),
 \end{align}
where $\ell = [x^{1-\epsilon}]+1$.  First, we shall try to find an estimate for $S_2(x^2)$.  Using the construction \eqref{M(l,n)} of $M(l;n)$,  for an integer $N > \ell$,   one can check that
\begin{align}\label{sum upto N}
\sum_{n=\ell}^{N} \frac{\mu(n)}{n^k} \exp\left({-\frac{x^2}{n^2}}\right) & = \frac{\mu(\ell)}{\ell^k} \exp\left({-\frac{x^2}{\ell^2}}\right) + \sum_{ \ell < n \leq N} \left(  M(l; n) - M(l; n-1) \right) \exp\left({-\frac{x^2}{n^2}}\right) \nonumber \\
& = \sum_{n=\ell}^{N-1} M(l;n) \left(e^{-\frac{x^2}{n^2}}- e^{-\frac{x^2}{(n+1)^2}}   \right) + M(l; N) e^{- \frac{x^2}{N^2} }.
\end{align}
Letting $N \rightarrow \infty$ in \eqref{sum upto N} and utilizing \eqref{bound for M(l,n)}, we arrive at 
\begin{align}
S_2(x^2)= S_3(x^2) + O_{\epsilon}\left(\ell^{\frac{1}{2}-k+\epsilon} \right),  \label{S_2}
\end{align}
where 
\begin{align*}
S_3(x^2) :=  \sum_{n=\ell}^{\infty} M(l;n) \left(e^{-\frac{x^2}{n^2}}- e^{-\frac{x^2}{(n+1)^2}}   \right). 
\end{align*}
To estimate $S_3(x^2)$,  we consider a function $f(y)= e^{-\frac{x^2}{y^2}}$.  Using mean value theorem,  one can find $y_n  \in (n,  n+1)$ such that $f(n) - f(n+1) =- f'(y_n) = - \frac{2 x^2}{y_n^3} e^{-\frac{x^2}{y_n^2}}$.  Utilizing this fact and together with \eqref{bound for M(l,n)},  we get
\begin{align*}
|S_3(x^2)| \ll_{\epsilon} \ell^{\frac{1}{2}-k+\epsilon} \sum_{n=\ell}^{\infty}   \frac{2 x^2}{y_n^3} e^{-\frac{x^2}{y_n^2}}.  
\end{align*}
Further,  we divide the sum into two parts in the following way,  
\begin{align}\label{S_3}
|S_3(x^2)| \ll_{\epsilon} \ell^{\frac{1}{2}-k+\epsilon} \left(  \sum_{n=\ell}^{ [ x] } + \sum_{n= [x]+1}^\infty \right)  \frac{x^2}{y_n^3} e^{-\frac{x^2}{y_n^2}}.
\end{align}
The finite sum can be trivially bounded by
\begin{align}\label{S_4}
\sum_{n=\ell}^{[x]} \frac{x^2}{y_n^3} e^{-\frac{x^2}{y_n^2}} \ll \sum_{n = \ell}^\infty \frac{x^2}{n^3} \ll \frac{x^2}{l^2} \ll x^{2 \epsilon},
\end{align}
as $\ell = [x^{1-\epsilon}] + 1$. 
Now the infinite part can be bounded by
\begin{align} \label{S_5}
\sum_{n= [x]+1}^\infty   \frac{x^2}{y_n^3} e^{-\frac{x^2}{y_n^2}} \ll \sum_{n >  x} \frac{x^2}{n^3} \ll O(1).
\end{align}
Here and in the previous inequality,  we have used the fact that $\sum_{n>x}  n^{-s}=O(x^{1-s})$ for $\Re(s)>1$.  Plugging \eqref{S_4} and \eqref{S_5} in \eqref{S_3}, we obtain
\begin{align*}
|S_3(x^2)| \ll_{\epsilon} x^{\frac{1}{2}-k+\epsilon}.
\end{align*}
Finally,  substitute this bound of $S_3(x^2)$ in \eqref{S_2} to see
\begin{align}\label{final bound of S_2}
|S_2(x^2)| = O_{\epsilon}  \left( x^{\frac{1}{2}-k+\epsilon}  \right).
\end{align}
Now we shall concentrate on $S_1(x^2)$.  We write
\begin{align*}
|S_1(x^2)| & =\left| \sum_{n=1}^{\ell-1}  \frac{\mu(n)}{n^k} \exp\left({-\frac{x^2}{n^2}}\right) \right|  
\leq  \sum_{n=1}^{\ell-1} \exp\left({-\frac{x^2}{l^2}}\right).
\end{align*}
This implies that 
\begin{align}\label{final bound of S_1}
S_1(x^2) = O \left( x^{1-\epsilon} \exp( - x^{2\epsilon} ) \right),
\end{align}
as $\ell = [x^{1-\epsilon}] + 1$.  In view of \eqref{final bound of S_1} and \eqref{final bound of S_2},  one can observe that the bound for $S_1(x^2)$ goes to zero much faster than the bound for $S_2(x^2)$ as $x \rightarrow \infty$.  Therefore,  \eqref{P_k(x^2) interms of S_1 and S_2} yields that
\begin{align*}
P_k( x^2) = O_{\epsilon}  \left( x^{\frac{1}{2}-k+\epsilon}  \right).
\end{align*}
Finally,  replacing $x$ by $\sqrt{x}$,  one can deduce \eqref{AGM_bound}.
Now assuming \eqref{AGM_bound}, we shall try to show that the Riemann hypothesis is true.  From Lemma \ref{Mellin_P_k(x)},  we know that 
\begin{align}\label{Lemma_Mellin_P_k(x)}
 \zeta(2s+k) \int_{0}^{\infty} x^{-s-1} P_k(x) {\rm d}x = \Gamma(-s)
\end{align}
valid in the region $\frac{1-k}{2} <\Re(s) < 1$, except at $s=0$. 
Next, we try to extend the region of validity of this lemma on the left half plane.  
For a large positive integer $N$,  we separate the above integral in the following way:
\begin{align*}
 \zeta(2s+k) \left( \int_{0}^{N} +  \int_{N}^{\infty}  \right) x^{-s-1}  P_k(x) {\rm d}x = \Gamma(-s)
\end{align*}
Employing the bound \eqref{AGM_bound} for $P_k(x)$,  one can easily show that the unbounded part of the above integration is finite if $\Re(s)> \frac{1}{4}-\frac{k}{2}$.  This will allow us to extend the analyticity of the identity \eqref{Lemma_Mellin_P_k(x)} in the region $ \frac{1}{4}-\frac{k}{2} <\Re(s) < \frac{1-k}{2}$.  Now we observe that the right side of \eqref{Lemma_Mellin_P_k(x)} is never zero in the region $ \frac{1}{4}-\frac{k}{2} <\Re(s) < \frac{1-k}{2}$,  which implies that $\zeta(2s+k)$ has no zero in the same region.
Equivalently,  we can say that $\zeta(s)$ has no zero in the region $ \frac{1}{2}< \Re(s) <1$ and thus the functional equation of $\zeta(s)$ implies that it has no zero $0< \Re(s) <\frac{1}{2}$.  This proves that all the non-trivial zeros of $\zeta(s)$ must lie on $\Re(s)=\frac{1}{2}$.  

\end{proof}

\section{Final Remarks}
In the current paper,  we established a one-variable generalization of the identity \eqref{Hardy-Littlewood}. 
Hardy and Littlewood \cite[p.~161]{HL-1916} pointed out that the absolute convergence of the series  in \eqref{Hardy-Littlewood},  over the non-trivial zeros of $\zeta(s)$, is not so immediate even after the assumption of the Riemann hypothesis.  We remark that the identity \eqref{Hardy-Littlewood} and our identity \eqref{AGM} are both valid under the assumption of the simplicity hypothesis of the non-trivial zeros of $\zeta(s)$ and we do not require the assumption of the Riemann hypothesis.  In \cite{JMS},  authors explained in more details why the convergence of this kind of series with $\zeta'(\rho)$ in the denominator is more delicate.
Another important observation is that the series 
$ \sum_{n=1}^{\infty} \frac{\mu(n)}{n^k} \exp\left({-\frac{x}{n^2}}\right) 
$
is absolutely convergent for any $k>1$,  whereas for $k=1$ the convergence is not absolute.  
 
 The identity \eqref{Hardy-Littlewood} has been studied by many mathematicians in various directions.  A connection with Wiener's Tauberian theory and the Fourier reciprocity has been established by  Bhaskaran \cite{bhas}.  A character analogue of  \eqref{Rama_Hardy_Little} was given by Dixit \cite{dixit12},  and further a one-variable generalization involving ${}_1F_{1}$ hypergeometric function was obtained by  Dixit, Roy and Zaharescu \cite{DRZ-character}.  Again,  Dixit et al. \cite{DRZ} established a generalization of \eqref{Rama_Hardy_Little} for Hecke eigenform.  In the same paper,  they also obtained a Riesz-type criteria for the Riemann hypothesis.  Recently,  an extension  for the Dedekind zeta function obtained by Dixit,  Gupta and Vatwani \cite{DGV21}.  They also established an equivalent criteria for the extended Riemann hypothesis.  Banerjee and Kumar \cite{BK21} derived an analogous identity corresponding to $L$-functions associated to the primtive cusp forms over $\Gamma_0(N)$ and gave an equivalent criteria for the grand Riemann hypothesis. 
 In a forthcoming paper,  we \cite{GM} are trying to find a character analogue of \eqref{AGM}, which will enable us to find an equivalent criteria for the generalized Riemann hypothesis for the Dirichlet $L$-function.

{\bf Acknowledgement:} The third author wants to thank SERB for the Start-Up Research Grant SRG/2020/000144.

\end{document}